\documentclass[11pt,leqno]{article}
\usepackage[margin=1in]{geometry} 
\usepackage{amssymb,amsfonts,amsmath,bbm,mathrsfs,stmaryrd,mathtools}
\usepackage{xcolor}
\usepackage{url}

\usepackage{graphicx}

\usepackage{accents}

\usepackage{extarrows}

\usepackage[shortlabels]{enumitem}
\usepackage{tensor}

\usepackage{xr}
\usepackage[T1]{fontenc}
\usepackage[utf8]{inputenc}

\usepackage[colorlinks,
linkcolor=black!75!red,
citecolor=blue,
pdftitle={},
pdfproducer={pdfLaTeX},
pdfpagemode=None,
bookmarksopen=true,
bookmarksnumbered=true,
backref=page]{hyperref}

\usepackage{tikz}
\usetikzlibrary{arrows,calc,decorations.pathreplacing,decorations.markings,decorations.shapes,intersections,shapes.geometric,through,fit,shapes.symbols,positioning,decorations.pathmorphing}

\makeatletter
\newlength\zig@L
\newlength\zig@La
\newlength\zig@Lb

\newcommand{\xzigrightarrow}[2][]{%
  \mathrel{%
    \settowidth{\zig@La}{$\scriptstyle #2$}%
    \settowidth{\zig@Lb}{$\scriptstyle #1$}%
    \zig@L=\zig@La\relax
    \ifdim\zig@Lb>\zig@L \zig@L=\zig@Lb\fi
    \advance\zig@L by 2.2em\relax
    \tikz[baseline=-0.65ex]{%
      \draw[->,
            line cap=round,
            decorate,
            decoration={zigzag,segment length=4pt,amplitude=1.1pt}]%
        (0,0) -- (\zig@L,0)
        node[midway,above=2pt] {$\scriptstyle #2$}%
        \if\relax\detokenize{#1}\relax\else
          node[midway,below=2pt] {$\scriptstyle #1$}%
        \fi
      ;
    }%
  }%
}
\makeatother

\makeatletter
\newcommand{\squigjoin}{1mu} 

\def\sqleft@{\sim}                    
\def\sqmid@{\sim\mkern-\squigjoin}    

\def\rightsquigarrowfill@{%
  \arrowfill@{\sqleft@}{\sqmid@}{\mkern-4mu\succ}%
}

\newcommand{\xrightsquigarrow}[2][]{%
  \ext@arrow 0359\rightsquigarrowfill@{#1}{#2}%
}
\makeatother

\makeatletter
\newcommand*\circled[1]{\tikz[baseline=(char.base)]{
    \node[shape=circle, draw, inner sep=0pt, 
    minimum height={\f@size},] (char) {\vphantom{WAH1g}#1};}}
\makeatother

\makeatletter
\DeclareRobustCommand\widecheck[1]{{\mathpalette\@widecheck{#1}}}
\def\@widecheck#1#2{%
    \setbox\z@\hbox{\m@th$#1#2$}%
    \setbox\tw@\hbox{\m@th$#1%
       \widehat{%
          \vrule\@width\z@\@height\ht\z@
          \vrule\@height\z@\@width\wd\z@}$}%
    \dp\tw@-\ht\z@
    \@tempdima\ht\z@ \advance\@tempdima2\ht\tw@ \divide\@tempdima\thr@@
    \setbox\tw@\hbox{%
       \raise\@tempdima\hbox{\scalebox{1}[-1]{\lower\@tempdima\box
\tw@}}}%
    {\ooalign{\box\tw@ \cr \box\z@}}}
\makeatother

\usepackage{braket}

\usepackage[amsmath,thmmarks,hyperref]{ntheorem}
\usepackage{cleveref}

\newcommand\nthalias[1]{\AddToHook{env/#1/begin}{\crefalias{lemma}{#1}}}

\nthalias{definition}
\nthalias{example}
\nthalias{examples}
\nthalias{remark}
\nthalias{remarks}
\nthalias{convention}
\nthalias{notation}
\nthalias{construction}
\nthalias{sketch}
\nthalias{theoremN}
\nthalias{propositionN}
\nthalias{corollaryN}
\nthalias{lemma}
\nthalias{proposition}
\nthalias{corollary}
\nthalias{theorem}
\nthalias{conjecture}
\nthalias{question}
\nthalias{assumption}

\creflabelformat{enumi}{#2#1#3}

\crefname{section}{Section}{Sections}
\crefformat{section}{#2Section~#1#3} 
\Crefformat{section}{#2Section~#1#3} 

\crefname{subsection}{\S}{\S\S}
\AtBeginDocument{%
  \crefformat{subsection}{#2\S#1#3}%
  \Crefformat{subsection}{#2\S#1#3}%
}

\crefname{subsubsection}{\S}{\S\S}
\AtBeginDocument{%
  \crefformat{subsubsection}{#2\S#1#3}%
  \Crefformat{subsubsection}{#2\S#1#3}%
}

\theoremstyle{plain}

\newtheorem{lemma}{Lemma}[section]
\newtheorem{proposition}[lemma]{Proposition}
\newtheorem{corollary}[lemma]{Corollary}
\newtheorem{theorem}[lemma]{Theorem}
\newtheorem{conjecture}[lemma]{Conjecture}

\theoremstyle{plain}
\theoremnumbering{Alph}

\theoremstyle{plain}
\theorembodyfont{\upshape}
\theoremsymbol{\ensuremath{\blacklozenge}}

\newtheorem{remark}[lemma]{Remark}
\newtheorem{remarks}[lemma]{Remarks}

\newtheorem{notation}[lemma]{Notation}

\crefname{definition}{definition}{definitions}
\crefformat{definition}{#2definition~#1#3} 
\Crefformat{definition}{#2Definition~#1#3} 

\crefname{ex}{example}{examples}
\crefformat{example}{#2example~#1#3} 
\Crefformat{example}{#2Example~#1#3} 

\crefname{exs}{example}{examples}
\crefformat{examples}{#2example~#1#3} 
\Crefformat{examples}{#2Example~#1#3} 

\crefname{remark}{remark}{remarks}
\crefformat{remark}{#2remark~#1#3} 
\Crefformat{remark}{#2Remark~#1#3} 

\crefname{remarks}{remark}{remarks}
\crefformat{remarks}{#2remark~#1#3} 
\Crefformat{remarks}{#2Remark~#1#3} 

\crefname{convention}{convention}{conventions}
\crefformat{convention}{#2convention~#1#3} 
\Crefformat{convention}{#2Convention~#1#3} 

\crefname{notation}{notation}{notations}
\crefformat{notation}{#2notation~#1#3} 
\Crefformat{notation}{#2Notation~#1#3} 

\crefname{table}{table}{tables}
\crefformat{table}{#2table~#1#3} 
\Crefformat{table}{#2Table~#1#3}

\crefname{lemma}{lemma}{lemmas}
\crefformat{lemma}{#2lemma~#1#3} 
\Crefformat{lemma}{#2Lemma~#1#3} 

\crefname{proposition}{proposition}{propositions}
\crefformat{proposition}{#2proposition~#1#3} 
\Crefformat{proposition}{#2Proposition~#1#3} 

\crefname{propositionN}{proposition}{propositions}
\crefformat{propositionN}{#2proposition~#1#3} 
\Crefformat{propositionN}{#2Proposition~#1#3} 

\crefname{corollary}{corollary}{corollaries}
\crefformat{corollary}{#2corollary~#1#3} 
\Crefformat{corollary}{#2Corollary~#1#3} 

\crefname{corollaryN}{corollary}{corollaries}
\crefformat{corollaryN}{#2corollary~#1#3} 
\Crefformat{corollaryN}{#2Corollary~#1#3} 

\crefname{theorem}{theorem}{theorems}
\crefformat{theorem}{#2theorem~#1#3} 
\Crefformat{theorem}{#2Theorem~#1#3} 

\crefname{theoremN}{theorem}{theorems}
\crefformat{theoremN}{#2theorem~#1#3} 
\Crefformat{theoremN}{#2Theorem~#1#3} 

\crefname{enumi}{}{}
\crefformat{enumi}{#2#1#3}
\Crefformat{enumi}{#2#1#3}

\crefname{assumption}{assumption}{Assumptions}
\crefformat{assumption}{#2assumption~#1#3} 
\Crefformat{assumption}{#2Assumption~#1#3} 

\crefname{construction}{construction}{Constructions}
\crefformat{construction}{#2construction~#1#3} 
\Crefformat{construction}{#2Construction~#1#3} 

\crefname{sketch}{sketch}{Sketches}
\crefformat{sketch}{#2sketch~#1#3} 
\Crefformat{sketch}{#2Sketch~#1#3} 

\crefname{question}{question}{Questions}
\crefformat{question}{#2question~#1#3} 
\Crefformat{question}{#2Question~#1#3} 

\crefname{equation}{}{}
\crefformat{equation}{(#2#1#3)} 
\Crefformat{equation}{(#2#1#3)}

\numberwithin{equation}{section}

\theoremstyle{nonumberplain}
\theoremsymbol{\ensuremath{\blacksquare}}

\newtheorem{proof}{Proof}
\newcommand\pf[1]{\newtheorem{#1}{Proof of \Cref{#1}}}

\newcommand\bC{{\mathbb C}}

\newcommand\bG{{\mathbb G}}
\newcommand\bH{{\mathbb H}}

\newcommand\bK{{\mathbb K}}

\newcommand\bP{{\mathbb P}}
\newcommand\bQ{{\mathbb Q}}

\newcommand\bS{{\mathbb S}}

\newcommand\bZ{{\mathbb Z}}

\newcommand\cC{{\mathcal C}}

\newcommand\cP{{\mathcal P}}

\newcommand\wh{\widehat}

\DeclareMathOperator{\im}{\mathrm{im}}

\newcommand{\cat}[1]{\textsc{#1}}

\newcommand\spr[1]{\cite[\href{https://stacks.math.columbia.edu/tag/#1}{Tag {#1}}]{stacks-project}}
\newcommand{\qedhere}{\mbox{}\hfill\ensuremath{\blacksquare}}

\newcommand{\comment}[1]{}

\newcommand{\xrightarrowdbl}[2][]{%
  \xrightarrow[#1]{#2}\mathrel{\mkern-14mu}\rightarrow
}

\title{Nullhomotopic equivariant shifts and Borsuk-Ulam-compatible compact abelian groups}
\author{Alexandru Chirvasitu and Alexander D{\k{a}}browski}

\begin{document}

\date{}

\newcommand{\Addresses}{{
  \bigskip
  \footnotesize

  \textsc{Department of Mathematics, University at Buffalo}
  \par\nopagebreak
  \textsc{Buffalo, NY 14260-2900, USA}  
  \par\nopagebreak
  \textit{E-mail address}: \texttt{achirvas@buffalo.edu}

  \medskip
  
  \textit{E-mail address}: \texttt{alex@dabrow.cc}

}}

\maketitle

\begin{abstract}
  We disprove the classical version of the Borsuk-Ulam-flavored conjecture, due to Baum-D{\k{a}}browski-Hajac, that free actions $X\times \mathbb{G}\to X$ on non-empty compact Hausdorff spaces by non-trivial compact groups admit no equivariant maps $X*\mathbb{G}\to X$. Specifically, among non-trivial compact abelian $\mathbb{G}$, the conjecture fails precisely when $\mathbb{G}$ is torsion-free profinite.
\end{abstract}


\vspace{.5cm}

\noindent{MSC 2020: 22C05; 54H15; 57S10; 57S17; 55R10; 55S35; 55R37; 55R40
  
  
}


\section*{Introduction}

The object of the present note is the purely classical version of \cite[Conjecture 2.3, Type 1]{MR3469159}: writing
\begin{equation*}
  X*Y
  :=
  X\times Y\times [0,1]
  \bigg/\left(
    \begin{gathered}
      (x,y,0)\sim (x',y,0)\\
      (x,y,1)\sim (x,y',1)
    \end{gathered}
  \right)
\end{equation*}
for the \emph{topological join} \cite[\S 14.4.3]{td_alg-top} of two spaces,

\begin{conjecture}[BDH]\label{cj:bdh.t1}
  For a free action $X\circlearrowleft \bG$ by a non-trivial compact (Hausdorff) group on a non-empty compact Hausdorff space, there are no $\bG$-equivariant maps $X*\bG\to X$.  \qedhere
\end{conjecture}

The statement is ``classical'' in the sense that loc. cit. is rather concerned with \emph{quantum} spaces: a possibly non-commutative $C^*$-algebra playing the role of the function algebra $C(X)$ and a \emph{compact quantum group} \cite[Definition 1.1.1]{NeTu13} in place of $\bG$.

The case $X:=\bG^{*n}$ recovers one form of the celebrated \emph{Borsuk-Ulam theorem}: \cite[Theorem 6.2.5]{mat_bu} for finite $\bG$, an unpublished result of Bestvina-Edwards for profinite $\bG$ and \cite[Theorem 1.1]{cdt_be-hs} for arbitrary compact $\bG$. The problem is also adjacent to the Hilbert-Smith conjecture \cite[Conjecture 17.0.2]{tao_h5_2014}, with a weak version thereof reducible (e.g. \cite[\S 4]{cdt_be-hs}) to the $\bG:=\bZ_p$ case ($(\bZ_p,+)$ denoting the $p$-adic integers) of \Cref{cj:bdh.t1}.

The main aim of the current note is to disprove the conjecture in the stated generality: \emph{all} $p$-adic-integer groups $\bZ_p$ provide counterexamples (\Cref{cor:higher.spheres}). Moreover, the compact abelian groups for which the conjecture fails are precisely the torsion-free profinite ones:

\begin{theorem}\label{th:cj1.holds.when}
  The following conditions on a non-trivial compact abelian group $\bG$ are equivalent.
  \begin{enumerate}[(a),wide]
  \item\label{item:th:cj1.holds.when:prod.3sph} $\bG$ admits a free action on a $3$-sphere power $X:=\left(\bS^{3}\right)^{S}$ with an equivariant map $X* \bG\to X$.
    
  \item\label{item:th:cj1.holds.when:prod.sph} $\bG$ admits a free action on a sphere power $X:=\left(\bS^{2n-1}\right)^{S}$, $n\ge 2$ with an equivariant map $X* \bG\to X$.

  \item\label{item:th:cj1.holds.when:cpct} $\bG$ admits a free action on a non-empty compact Hausdorff space $X$ with an equivariant map $X* \bG\to X$. 

  \item\label{item:th:cj1.holds.when:tors.fr} $\bG$ is torsion-free and profinite.

  \item\label{item:th:cj1.holds.when:p.adic} $\bG$ is a product of $p$-adic groups $\bZ_p$ for various primes $p$. 
  \end{enumerate}
\end{theorem}

The counterexamples consist of actions $\bG$ on products $X:=\prod_i Y_i$, each factor constituting an action $Y_i\circlearrowleft \bG_i$ by a quotient $\bG\xrightarrowdbl{}\bG_i$ for a cofiltered-limit realization $\bG\cong\varprojlim_i \bG_i$. The abstract gadget is isolated as \Cref{th:act.prod.fr} below, and requires nullhomotopic equivariant maps between (some of) the individual factors $Y_i$ which aggregate to a nullhomotopic equivariant self-map of the product $X$. That this is all available for $\bG:=\bZ_p$ is noted in \cite[Theorem 3.15]{MR3990731} and its proof: there are $\left(\bZ/p^{2k},\bZ/p^k\right)$-equivariant nullhomotopic maps both
\begin{itemize}[wide]
\item  $\left(\bZ/p^{2k}\right)^{*m}\to \left(\bZ/p^{k}\right)^{*m}$ for $k\ge 1$, $m\ge 4$ and primes $p$;  

\item and $\bS^3\to \bS^3$ (for the standard simultaneous-scaling action on $\bS^3\subset \bC^2$). 
\end{itemize}

\subsection*{Acknowledgments}

We are grateful for helpful input from L. D{\k{a}}browski. 

The initial counterexample was identified by ChatGPT Astra and verified by Claude Fable during conversations with the second author. The first author then clarified the construction and extended it substantially into the full results of this paper.

\section{Free actions on products and BDH counterexamples}\label{se:prods.bdh.counter}

Recall \cite[Lemma 4.1]{MR4523296}: for continuous right group actions $X,Y\circlearrowleft \bG$, an equivariant map $X*\bG\to Y$ exists if and only if a nullhomotopic $\bG$-map $h\in Y^X$ does; 
\begin{equation*}
  X\times \left(I:=[0,1]\right)
  \xrightarrow{\quad H\quad}
  Y
  ,\quad
  H(-,0)=h
  ,\quad
  H(-,1)\equiv y_0\in Y
\end{equation*}
will produce a $\bG$-equivariant map
\begin{equation*}
  X*\bG
  \ni
  tg+(1-t)x
  \xmapsto{\quad}
  H\left(xg^{-1},t\right)g
  \in Y.
\end{equation*}

A more specific observation would run as in \Cref{pr:eqvr.map.nullhot} below. We will denote by $\cat{Top}^{\bG}$ the category of topological (right) $\bG$-spaces for a topological group $\bG$ (abbreviating $\cat{Top}:=\cat{Top}^{\{1\}}$ when convenient). 

\begin{proposition}\label{pr:eqvr.map.nullhot}
  Let $\bG$ be a locally compact Hausdorff group and $X\in \cat{Top}^{\bG}$. The $\bG$-space $X*\bG$ with diagonal action represents the functor
  \begin{equation*}
    \cat{Top}^{\bG}
    \ni
    Y
    \xmapsto{\quad}
    \cat{Top}_{X\circlearrowleft \bG}(X*\bG,Y)
    :=
    \left\{
      \cC X
      \xrightarrow[\quad\text{continuous}\quad]{\quad f\quad}
      Y
      \quad:\quad
      f|_{X}\in \cat{Top}^{\bG}
    \right\}
    \in
    \cat{Set},
  \end{equation*}
  where $\cC X:=X\times [0,1]/X\times \{1\}$ is the \emph{cone} on $X$ \cite[p. 9]{hatch_at} and $X\cong X\times \{0\}\subseteq \cC X$ is its base. 
\end{proposition}
\begin{proof}
  There is a natural isomorphism
  \begin{equation}\label{eq:fgt.g}
    \cat{Top}\left(\cC X,Y\right)
    \cong
    \cat{Top}^{\bG}\left(\cC X\times \bG,Y\right)    
  \end{equation}
  induced by the adjunction between the forgetful functor $\cat{Top}^{\bG}\to \cat{Top}$ and $-\times \bG$, restricting to
  \begin{equation*}
    \cat{Top}_{X\circlearrowleft \bG}\left(\cC X,Y\right)
    \cong
    \cat{Top}^{\bG}\left(X*^r\bG,Y\right)    
  \end{equation*}
  upon
  \begin{itemize}[wide]
  \item identifying
    \begin{equation*}
      X*\bG
      \cong
      \cC X\times \bG\bigg/\left((xg,h)\sim (x,gh),
        \begin{gathered}
          \forall x\in X\subseteq \cC X\\
          \forall g,h\in \bG
        \end{gathered}
        \right),
      \end{equation*}
    \item with ``$*^r$'' denoting the join with the right-hand translation (rather than diagonal) $\bG$-action. 
  \end{itemize}
  To conclude, apply the self-homeomorphism $(p,g)\mapsto (pg,g)$ of $\cC X\times \bG$, intertwining the right-hand and diagonal actions.
\end{proof}

In the particular case $X=Y$ nullhomotopic equivariant self-maps $h\in X^X$ can be produced by means of the device employed in the following statement.

\begin{theorem}\label{th:act.prod.fr}
  Suppose given
  \begin{itemize}[wide]
  \item a topological group $\bG=\varprojlim_I \bG_i$ exhibited as a \emph{cofiltered limit} \spr{04AZ} for
    \begin{equation*}
      \bG_j\xrightarrow{\quad\pi_{ij}\quad}\bG_i
      ,\quad
      j\ge i\in I;
    \end{equation*}

  \item free actions $Y_i\circlearrowleft \bG_{i}:\alpha_i$, $i\in I$;
    
  \item nullhomotopic maps $h_i\in Y_i^{Y_{\theta i}}$, $\pi_{i,\theta i}$-equivariant for a map $\theta\in I^I$, $\theta i\ge i$ in the sense that
    \begin{equation}\label{eq:transition.eqvrnt}
      \forall i
      \left(
        h_i\left(y\triangleleft\alpha_{\theta i}g\right)
        =
        h_i(y)\triangleleft\alpha_i\pi_{i,\theta i}(g)
      \right).
    \end{equation}
  \end{itemize}
  This will provide
  \begin{itemize}[wide]
  \item a free $\bG$-action on $X:=\prod_i Y_i$, with $\bG$ acting on the individual $i$-indexed factor $Y_i$ via $\alpha_i$;
  \item and a nullhomotopic $\bG$-equivariant self-map
    \begin{equation*}
      \prod_i Y_i
      \ni
      \left(y_i\right)_{i\in I}
      \xmapsto{\quad h\quad}
      \left(h_i(y_{\theta i})\right)_{i\in I}
      \in
      \prod_i Y_i.
    \end{equation*}
  \end{itemize}
\end{theorem}
\begin{proof}
  Freeness for the (plainly continuous) product action
  \begin{equation*}
    \bG
    \xrightarrow{\quad\prod_i \alpha_i\quad}
    \prod_i Y_i^{Y_i}
  \end{equation*}
  follows from $\bG=\varprojlim_I \bG_i$. Aggregating
  \begin{equation*}
    Y_{\theta i}\times [0,1]
    \xrightarrow{\quad H_i\quad}
    Y_i
    ,\quad
    H_i(-,0)\equiv h_i
    ,\quad
    H_i(-,1)\equiv y_{i,0}\in  Y_i
  \end{equation*}
  into a single homotopy
  \begin{equation*}
    \left(\prod_i Y_i\right)\times [0,1]
    \ni
    \left((y_i)_{i\in I},t\right)
    \xmapsto{\quad H\quad}    
    \left(H_i\left(y_{\theta i},t\right)\right)_i
    \in
    \prod_i Y_i
  \end{equation*}
  will finish the proof. 
\end{proof}

The case of interest is $\bG:=\bZ_p\cong \varprojlim_{n\in \bZ_{\ge 0}}\bZ/p^{2^n}$, with $\theta n:=n+1$. The existence of nullhomotopic $h_n$ for $Y_i:=\bS^3$ is noted in the course of the proof of \cite[Theorem 3.15]{MR3990731}. More generally:

\begin{corollary}\label{cor:higher.spheres}
  For every prime $p$ and every $n\in \bZ_{\ge 2}$ there is a free action of the additive $p$-adic group $\bZ_p$ on $X:=\left(\bS^{2n-1}\right)^{\bZ_{\ge 0}}$ admitting an equivariant map $X*\bZ_p\to X$. 
\end{corollary}
\begin{proof}
  Embed
  \begin{equation*}
    \bS^{2n-1}
    =
    \left\{(z_1,\cdots,z_n)\in \bC^n\ :\ \sum_i |z_i|^2=1\right\}
    \subset
    \bC^n
  \end{equation*}
  and consider the degree-$q^n$ self-map
  \begin{equation*}
    \bS^{2n-1}
    \ni
    (r_i u_i)_{i=1}^n
    \xmapsto{\quad}
    (r_i u_i^q)_{i=1}^n
    \in
    \bS^{2n-1}
    \quad
    \left(
      \text{polar decomposition }\bC\ni z_i=r_i u_i
    \right)
  \end{equation*}
  for any $p$-power $q$, equivariant with respect to the canonical surjection $\bZ/q^2\xrightarrowdbl{} \bZ/q$ in the sense of \Cref{eq:transition.eqvrnt}. Setting $q:=p^{2^i}$, said maps can be altered along the free $q^2$-sized orbits (cf. \cite[\S 4]{MR3803650}) to produce nullhomotopic self-maps $h_i$ of $\bS^{2n-1}$. The conclusion now follows from \Cref{th:act.prod.fr} on taking $Y:=\bS^{2n-1}$ and realizing $\bZ_p\cong \varprojlim_{i\in \bZ_{\ge 0}}\bZ/p^{2^i}$ with $\theta i:=i+1$.
\end{proof}

\begin{remarks}\label{res:why.3sph}
  \begin{enumerate}[(1),wide]
  \item The degree-modification procedure in the proof of \Cref{cor:higher.spheres} requires that the orbit size $q^2$ divide the degree $q^n$ of the original map; the argument thus only goes through for odd spheres of dimension $\ge 3$.

  \item For $\bZ_p$, \Cref{th:act.prod.fr} can also be instantiated with a different choice of actions $Y_i \circlearrowleft \bG_i$: recall \cite[Theorem 3.15]{MR3990731} the existence of maps
    \begin{equation}\label{eq:z2p.2.zp}
      \left(\bZ/p^{2k}\bZ\right)^{*(m+1)}
      \xrightarrow[\quad\text{$\left(\bZ/p^{2k},\bZ/p^k\bZ\right)$-equivariant}\quad]{\quad}
      \left(\bZ/p^{k}\bZ\right)^{*m}
      \quad
      \begin{gathered}
        k\ge 1\\
        m\ge 4\\
        \text{prime }p\ge 2,
      \end{gathered}
    \end{equation}
    whence (by \Cref{pr:eqvr.map.nullhot}) nullhomotopic $\left(\bZ/p^{2k}\bZ\right)^{*m}\to \left(\bZ/p^{k}\bZ\right)^{*m}$, equivariant with respect to the surjection $\bZ/p^{2k}\xrightarrowdbl{}\bZ/p^k\bZ$. We can thus take
    \begin{equation*}
      (I,\le):=\left(\bZ_{\ge 0},\le\right)
      ,\quad
      \bG_i:=\bZ/p^{2^i k}\bZ
      ,\quad
      Y_i:= E_{m-1}\bG_i:= \bG_i^{*m}
      ,\quad
      i\xmapsto{\ \theta\ }i+1
    \end{equation*}
    for fixed $p$, $k$ and $m$ as in the already-referenced \cite[Theorem 3.15]{MR3990731}. 
  \end{enumerate}  
\end{remarks}

\cite[Theorem 3.15]{MR3990731}'s approach to \Cref{eq:z2p.2.zp} in fact uses the $\left(\bZ/p^{2k},\bZ/p^k\right)$-equivariant self-maps on $\bS^{3}$ mentioned in the proof of \Cref{cor:higher.spheres}. An alternative cohomological, obstruction-theoretic argument is available affording an application of \Cref{th:act.prod.fr}; we provide it in \Cref{th:z2p.2.zp} for completeness (note also the slight improvement over  \cite[Theorem 3.15]{MR3990731}, to the sharp range $m\ge 3$). 


For a topological group $\bG$, write
\begin{equation*}
  E\bG:=\bigcup_n \left(E_n\bG:=\bG^{*(n+1)}\right)
  ,\quad
  B\bG:=E\bG/\bG=\bigcup_n \left(B_n\bG:=E_n\bG/\bG\right)
\end{equation*}
for the total and base space of Milnor's \emph{universal bundle} discussed in \cite[\S 14.4.3]{td_alg-top} (and its $n$-dimensional truncations indicated by the subscripts ``$n$''). We refer the reader to \cite[\S 2]{MR2475974} and/or \cite[\S 2.3]{MR3426700}.

\begin{notation}\label{not:bs}
  For a finite group $\bG$, write $\beta_{\bG}\in H^1(\bG,K_{\bG})$ for the \emph{Berstein-Schwarz class} \cite[\S 2.3]{MR3426700} attached (i.e. classifying) the extension 
  \begin{equation*}
    0\to
    K_{\bG}
    \lhook\joinrel\xrightarrow{\quad}
    \bZ\bG
    \xrightarrowdbl[\quad\text{counit}\quad]{\quad\varepsilon\quad}
    \bZ
    \to 0
  \end{equation*}
  in the category of $\bG$-modules. 
\end{notation}

\begin{theorem}\label{th:z2p.2.zp}
  The following conditions on $d\in \bZ_{\ge 1}$ are equivalent.
  \begin{enumerate}[(a),wide]
  \item\label{item:th:z2p.2.zp:d3} $d\ge 3$.

  \item\label{item:th:z2p.2.zp:res.n.n2} For every $n\in \bZ_{\ge 1}$ the group-cohomology restriction maps
    \begin{equation*}
      H^d(\bZ/n,\bullet)
      \xrightarrow{\quad}
      H^d(\bZ/n^2,\bullet)
    \end{equation*}
    vanish.

  \item\label{item:th:z2p.2.zp:res.n.n} For every $n\in \bZ_{\ge 1}$ maps
    \begin{equation*}
      H^d(\bZ/n,\bullet)
      \xrightarrow{\quad}
      H^d(\bZ/N,\bullet)
    \end{equation*}
    vanish for $N\in \bZ_{\ge 1}$ sufficiently large divisibility-wise. 

  \item\label{item:th:z2p.2.zp:res.n.prof} For every $n\in \bZ_{\ge 1}$ the maps
    \begin{equation*}
      H^d(\bZ/n,\bullet)
      \xrightarrow{\quad}
      H^d\left(\wh{\bZ},\bullet\right)
      \cong\varinjlim_N H^d\left(\bZ/N,\bullet^{N\wh{\bZ}}\right)
      ,\quad
      \wh{\bZ}:=\text{\emph{profinite completion} $\varprojlim_N \bZ/N$}
    \end{equation*}
    vanish.

  \item\label{item:th:z2p.2.zp:bg} Any of \Cref{item:th:z2p.2.zp:res.n.n2}-\Cref{item:th:z2p.2.zp:res.n.prof} applied only to the single class $\beta^d_{\bZ/n}$.
    
  \item\label{item:th:z2p.2.zp:eqv.n.n2} For every $n\in \bZ_{\ge 1}$ we have a map
    \begin{equation*}
      E_d\left(\bZ/n^2\right)
      =
      \left(\bZ/n^2\right)^{*(d+1)}
      \xrightarrow[\
      \text{$\left(\bZ/n^2,\bZ/n\right)$-equivariant}\ ]{\quad}
      \left(\bZ/n\right)^{*d}
      =
      E_{d-1}\left(\bZ/n\right).
    \end{equation*}
    
  \item\label{item:th:z2p.2.zp:eqv.n.n} For every $n\in \bZ_{\ge 1}$ and every $N\in \bZ_{\ge 1}$ sufficiently large divisibility-wise, we have a map
    \begin{equation}\label{eq:em.2.em1}
      E_d\left(\bZ/N\right)
      =
      \left(\bZ/N\right)^{*(d+1)}
      \xrightarrow[\
      \text{$\left(\bZ/N,\bZ/n\right)$-equivariant}\ ]{\quad}
      \left(\bZ/n\right)^{*d}
      =
      E_{d-1}\left(\bZ/n\right).
    \end{equation}
  \end{enumerate}
\end{theorem}
\begin{proof}
  The implications 
  \begin{equation*}
    \begin{tikzpicture}[>=stealth,auto,baseline=(current  bounding  box.center)]
      \path[anchor=base] 
      (0,0) node (l) {\Cref{item:th:z2p.2.zp:d3}}
      +(2,1) node (mu) {\Cref{item:th:z2p.2.zp:eqv.n.n}}
      +(2,-1.5) node (md) {\Cref{item:th:z2p.2.zp:res.n.n2}}
      +(4,1) node (ru) {\Cref{item:th:z2p.2.zp:eqv.n.n2}}
      +(4,-1) node (rd) {\Cref{item:th:z2p.2.zp:res.n.n}}
      +(6,0) node (r) {\Cref{item:th:z2p.2.zp:res.n.prof}}
      ;
      \draw[-implies,double equal sign distance] (l) to[bend left=6] node[pos=.5,auto] {$\scriptstyle \mathrm{scd}\left(\wh{\bZ}\right)=2$} node[pos=.5,auto,swap] {$\scriptstyle \text{\cite[\S I.3.2, Example (1)]{ser_gal}}$} (rd);
      \draw[-implies,double equal sign distance] (md) to[bend right=6] node[pos=.5,auto] {$\scriptstyle $} (rd);
      \draw[-implies,double equal sign distance] (rd) to[bend right=6] node[pos=.5,auto] {$\scriptstyle $} (r);
      \draw[-implies,double equal sign distance] (ru) to[bend right=6] node[pos=.5,auto] {$\scriptstyle $} (mu);
    \end{tikzpicture}
  \end{equation*}
  are in place, formally for most of the unadorned arrows and with ``$\mathrm{scd}$'' denoting \emph{strict cohomological dimension} \cite[\S I.3.2]{ser_gal}. Further:
  \begin{equation*}
    \begin{gathered}
      \forall\left(\bullet\in \left\{\text{\Cref{item:th:z2p.2.zp:res.n.n2},\Cref{item:th:z2p.2.zp:res.n.n},\Cref{item:th:z2p.2.zp:res.n.prof}}\right\}\right)
      \left(
        \bullet
        \xRightarrow[\quad]{\quad}
        \text{\Cref{item:th:z2p.2.zp:bg}}\bullet
      \right)
      \quad\text{and}\\
      \text{\Cref{item:th:z2p.2.zp:bg}\Cref{item:th:z2p.2.zp:res.n.n2}}
      \xRightarrow{\quad}
      \text{\Cref{item:th:z2p.2.zp:bg}\Cref{item:th:z2p.2.zp:res.n.n}}
      \xLeftrightarrow{\quad H^d\left(\wh{\bZ},\bullet\right)
        \cong\varinjlim_N H^d\left(\bZ/N,\bullet^{N\wh{\bZ}}\right)\quad}
      \text{\Cref{item:th:z2p.2.zp:bg}\Cref{item:th:z2p.2.zp:res.n.prof}},
    \end{gathered}
  \end{equation*}
  with \Cref{item:th:z2p.2.zp:bg}$\bullet$ meaning ``$\bullet$-flavored variant of \Cref{item:th:z2p.2.zp:bg}''. Conversely:
  \begin{equation*}
    \forall\left(\bullet\in \left\{\text{\Cref{item:th:z2p.2.zp:res.n.n2},\Cref{item:th:z2p.2.zp:res.n.n},\Cref{item:th:z2p.2.zp:res.n.prof}}\right\}\right)
    \left(
      \text{\Cref{item:th:z2p.2.zp:bg}}\bullet      
      \xRightarrow[\quad]{\quad}
      \bullet
    \right)
  \end{equation*}
  by the universality \cite[Corollary 3.5]{MR2475974} of the Berstein-Schwarz class. A few other implications will close the circle.

  \begin{enumerate}[label={},wide]

  \item\textbf{\Cref{item:th:z2p.2.zp:d3} $\Rightarrow$ \Cref{item:th:z2p.2.zp:res.n.n2}:} One elementary observation (appearing as \cite[Exercise 1 post \S Xiii.1]{ser_lf}, for instance) will be useful: having identified \cite[\S VIII.4]{ser_lf}
    \begin{equation*}
      \begin{aligned}
        H^{\text{even $\ge 2$}}(\bZ/n\cong\Braket{\sigma},A)
        &\cong
          \ker \left(\sigma-1)\right/\im \Sigma\\
        H^{\text{odd $\ge 1$}}(\bZ/n,A)
        &\cong
          \ker \Sigma/\im \left(\sigma-1\right)
      \end{aligned}
      ,\quad
      \Sigma:=\sum_{i=0}^{n-1}\sigma^i,
    \end{equation*}
    we have
    \begin{equation*}
      H^{d\ge 0}(\bZ/n,\bullet)
      \xrightarrow[\quad\text{restriction}\quad]{\quad \text{multiplication by }m^{\lfloor\frac d2\rfloor}\quad}
      H^d(\bZ/nm,\bullet).
    \end{equation*}
    Indeed, this is so
    \begin{itemize}[wide]
    \item for $d=0,1$ immediately;
    \item for the class
      \begin{equation}\label{eq:u.can.cls}
        u
        \in
        H^2\left(\bZ/N,\bZ\right)
        \cong
        H^1\left(\bZ/N,\bQ/\bZ\right)
      \end{equation}
      corresponding to the character sending a fixed generator to $\frac 1N$;
    \item hence also generally by periodicity \cite[\S VIII.4, Remark post Corollary to Proposition 6]{ser_lf}
      \begin{equation*}
        H^{d\ge 1}\left(\bZ/N,\bullet\right)
        \xrightarrow[\quad\cong \quad]{\quad u\cdot \quad}
        H^{d+2}\left(\bZ/N,\bullet\right).
      \end{equation*}
    \end{itemize}
    Restriction naturality in the form of
    \begin{equation*}
      \begin{tikzpicture}[>=stealth,auto,baseline=(current  bounding  box.center)]
        \path[anchor=base] 
        (0,0) node (l) {$H^{\bullet\ge 1}(\bG,I/A)$}
        +(3,.5) node (u) {$H^{\bullet}(\bH,I/A)$}
        +(3,-.5) node (d) {$H^{\bullet+1}(\bG,A)$}
        +(6,0) node (r) {$H^{\bullet+1}(\bH,A)$}
        ;
        \draw[->] (l) to[bend left=6] node[pos=.5,auto] {$\scriptstyle \mathrm{res}$} (u);
        \draw[->] (u) to[bend left=6] node[pos=.5,auto] {$\scriptstyle $} (r);
        \draw[->] (l) to[bend right=6] node[pos=.5,auto,swap] {$\scriptstyle \cong$} (d);
        \draw[->] (d) to[bend right=6] node[pos=.5,auto,swap] {$\scriptstyle \mathrm{res}$} (r);
      \end{tikzpicture}
    \end{equation*}
    for group morphisms $\bH\to \bG$ and embeddings $A\le I$ into injective $\bG$-modules reduces the problem to $d=3$ (or, what matters: odd $d$ only). Observe, to that end, that for any $\bZ/n$-module $A$
    \begin{equation*}
      \forall\left(a\in A\right)
      \left(
        \Sigma_{\bZ/n}a=0
        \xRightarrow{\quad}
        na
        =
        na-\Sigma_{\bZ/n}a
        =
        \sum_{i=0}^{n-1}(1-\sigma^i)a
        \in
        \im\left(\sigma-1\right)
      \right).
    \end{equation*}
    
  \item\textbf{\Cref{item:th:z2p.2.zp:res.n.prof} $\Rightarrow$ \Cref{item:th:z2p.2.zp:d3}:} Simply observe that \Cref{eq:u.can.cls}, regarded as an element of either $H^{1,2}$, does not restrict to 0:
    \begin{itemize}[wide]
    \item for $H^1$ because the character
      \begin{equation*}
        \bZ/n\cong\Braket{\sigma}
        \ni
        \sigma
        \xmapsto{\quad}
        \frac 1n
        \in \bQ/\bZ
      \end{equation*}
      will of course remain non-trivial upon restricting to $\wh{\bZ}$;

    \item and for $H^2$ because $\bZ$, regarded as an extension of $\bZ/n$ by $\bZ$, does not split when pulled back along $\bZ/nm\xrightarrowdbl{}\bZ/n$ no matter the $m$. 
    \end{itemize}

  \item\textbf{
      \Cref{item:th:z2p.2.zp:bg}\Cref{item:th:z2p.2.zp:res.n.n2} $\iff$ \Cref{item:th:z2p.2.zp:eqv.n.n2}
      and
      \Cref{item:th:z2p.2.zp:bg}\Cref{item:th:z2p.2.zp:res.n.n} $\iff$ \Cref{item:th:z2p.2.zp:eqv.n.n}
      :} Recall (e.g. \cite[Corollary 2.6]{MR3426700} and obstruction-theory naturality) that for a principal $\bG$-bundle $P\xrightarrowdbl{\pi}X$ with classifying map $X\xrightarrow{f} B\bG$ the pullback
    \begin{equation*}
      f^*\left(\beta_{\bG}^{M+1}\right)
      \in
      H^{M+1}\left(X;f^*K^{\otimes(M+1)}_{\bG}\right)
      ,\quad
      \left(\text{$K^{\otimes(M+1)}_{\bG}$ regarded as a bundle over $B\bG$}\right)
    \end{equation*}
    is the \emph{primary obstruction} \cite[p. 419]{hatch_at} to the existence of a $\bG$-equivariant map $P\to E_M\bG$ (equivalently, to $\pi$ being a pullback of $E_M\bG\xrightarrowdbl{} B_M\bG$). 

    For a general topological-group morphism $\bH\to \bK$ the classifying map of the principal $\bK$-bundle
    \begin{equation*}
      E\bH\times_{\bH}\bK
      \xrightarrowdbl{\quad}
      \left(E\bH\times_{\bH}\bK\right)/\bK
      \cong
      E\bH/\bH
      =
      B\bH
    \end{equation*}
    is classified precisely by the map $B\bH\to B\bK$ afforded by classifying-space functoriality. Applying this to $\bZ/N\xrightarrowdbl{\varphi} \bZ/n$,
    \begin{equation*}
      \dim B_d=d
      \xRightarrow{\quad}
      \left(\text{\Cref{eq:em.2.em1} exists $\xLeftrightarrow{\quad}$ }\left(B\varphi|_{B_d \bZ/N}\right)^*\beta^{d}_{\bZ/n}=0\right).
    \end{equation*}
    The long exact sequence of the pair $\left(B\bZ/N,\ B_{d}\bZ/N\right)$ further equates this to $B\varphi^*\beta^{d}_{\bZ/n}$ vanishing, completing the proof of the present branch for $d\ge 3$. To conclude, it remains to argue that for each claimed equivalence the two conditions involved fail simultaneously for $d=1,2$.

    On the one hand, \Cref{item:th:z2p.2.zp:bg}\Cref{item:th:z2p.2.zp:res.n.n2} and \Cref{item:th:z2p.2.zp:bg}\Cref{item:th:z2p.2.zp:res.n.n} do so fail by their already-noted respective equivalence to \Cref{item:th:z2p.2.zp:res.n.n2} and \Cref{item:th:z2p.2.zp:res.n.n} (Berstein-Schwarz universality) and the proof of \Cref{item:th:z2p.2.zp:res.n.prof} $\Rightarrow$ \Cref{item:th:z2p.2.zp:d3} above.

    On the other hand, consider a nullhomotopic
    \begin{equation*}
      \left(\bZ/nm\right)^{*d}
      \xrightarrow[\
      \text{$\left(\bZ/nm,\bZ/n\right)$-equivariant}\ ]{\quad n\ge 2,\quad d=1,2\quad}
      \left(\bZ/n\right)^{*d}.
    \end{equation*}
    For $d=1$ this is an equivariant surjection of discrete finite spaces carrying transitive actions, contradicting nullhomotopy (which would instead force constancy) as soon as $n\ge 2$. As to $d=2$, consider the case $n:=2$, $N:=2m$:
    \begin{itemize}[wide]
    \item $E_1\bZ/2$ is an antipodal circle (i.e. circle carrying the antipodal $\bZ/2$-action);
    \item the quotient $\left(E_1\bZ/2m\right)/\left(\bZ/m\right)$ consists of $m$ antipodal circles glued along four shared points constituting two pairs of antipodes;
    \item because in particular the latter space contains an antipodal circle, the circle-specific Borsuk-Ulam theorem and the torsion-freeness of the codomain provide a non-trivial composition
      \begin{equation*}
        \pi_1 E_1\bZ/2m
        \lhook\joinrel\xrightarrow{\quad\text{finite index}\quad}
        \pi_1 \left(E_1\bZ/2m\right)/\left(\bZ/m\right)
        \xrightarrow[\quad\text{\cite[Proposition 2.4.1(iii)]{mat_bu}}\quad]{\quad\text{non-trivial}\quad}
        \bZ\cong \pi_1\bS^1
      \end{equation*}
      as a consequence of any \Cref{eq:em.2.em1}. 
    \end{itemize}
  \end{enumerate}
\end{proof}

\Cref{th:z2p.2.zp} in hand, the argument proving \cite[Theorem 3.15]{MR3990731} can be run in reverse: appropriately equivariant maps between spheres exist as a consequence of having \Cref{eq:em.2.em1} available (rather than vice versa), independently of any sphere-specific degree considerations. Recall \cite[Definition 6.2.1]{mat_bu} that for finite non-trivial $\bG$ an \emph{$E_n\bG$-space} (\emph{$E_n$-space} with $\bG$ understood) is a finite free simplicial $\bG$-complex, $n$-dimensional and $(n-1)$-connected. 

\begin{corollary}\label{cor:sph.maps.conseq}
  For
  \begin{equation*}
    X\circlearrowleft \bZ/n^2
    \quad\text{and}\quad
    Y\circlearrowleft \bZ/n
    \quad
    \left(\text{$E_d$-spaces}\right)
    ,\quad
    n\in \bZ_{\ge 1},\ d\in \bZ_{\ge 3}
  \end{equation*}
  There are nullhomotopic $(\bZ/n^2,\bZ/n)$-equivariant maps $X\to Y$ or, equivalently, equivariant maps $X*\bZ/n^2\to Y$. 
\end{corollary}
\begin{proof}
  The equivalence is provided by \Cref{pr:eqvr.map.nullhot}. As to existence: $E_{d+1}\bZ/n^2$ and $Y$ are respectively \emph{$(d+1)$-} and \emph{$d$-universal} (\cite[\S 19.2]{steen_fib_1999} and \cite[Lemma 2.3]{MR77932}), being $d$- and $(d-1)$-connected; this provides equivariant maps
  \begin{equation*}
    X*\bZ/n^2
    \xrightarrow{\quad}
    E_{d+1}\bZ/n^2
    \quad\text{and}\quad
    E_d\bZ/n
    \xrightarrow{\quad}
    Y.
  \end{equation*}
  A $(\bZ/n^2,\bZ/n)$-equivariant map \Cref{eq:em.2.em1} will bridge the two, resulting in the sought-after equivariant composition.
\end{proof}

The techniques featuring above suffice to characterize those compact abelian groups for which \cite[Conjecture 2.3, Type 1]{MR3469159} fails to hold.

\pf{th:cj1.holds.when}
\begin{th:cj1.holds.when}
  We have
  \begin{equation*}
    \text{\Cref{item:th:cj1.holds.when:prod.3sph}}
    \xRightarrow{\quad\text{formal}\quad}
    \text{\Cref{item:th:cj1.holds.when:prod.sph}}
    \xRightarrow{\quad\text{formal}\quad}
    \text{\Cref{item:th:cj1.holds.when:cpct}}
    \xRightarrow{\quad\text{\cite[Proposition 4.1]{MR4523296}}\quad}
    \text{torsion-free}.
  \end{equation*}
  The equivalence \Cref{item:th:cj1.holds.when:tors.fr} $\iff$ \Cref{item:th:cj1.holds.when:p.adic} is standard structure theory (the forward implication carrying the substance): $\bG$ being torsion-free, the discrete Pontryagin dual $\wh{\bG}$ is \emph{divisible} \cite[Definition post Proposition 16]{mor-pont} by \cite[Theorem 3.1]{mor-pont} and hence 
  \begin{equation}\label{eq:ghat.decomp}
    \begin{gathered}
      \wh{\bG}
      \cong
      \bQ^{\oplus S}
      \oplus
      \bigoplus_{\text{primes }p}\left(\bZ/p^{\infty}\right)^{\oplus S_p}
      \quad
      \xRightarrow{\quad}
      \quad
      \bG
      \cong
      \wh{\bQ}^{S}
      \times
      \prod_{p}\bZ_p^{S_p}\\
      \text{sets }S,S_p
      ,\quad
      \bZ/p^{\infty}:=\left\{\text{$p$-power roots of unity}\right\}
    \end{gathered}
  \end{equation}
  by \cite[Theorem 4]{kap}, with the connected $\wh{\bQ}^S$ factor absent for profinite $\bG$. It thus remains to prove
  \begin{itemize}[wide]
  \item \Cref{item:th:cj1.holds.when:p.adic} $\Rightarrow$ \Cref{item:th:cj1.holds.when:prod.sph};

  \item and \Cref{item:th:cj1.holds.when:cpct} $\Rightarrow$ $\bG$ is profinite. 
  \end{itemize}
  The former is immediate, given the proof of \Cref{cor:higher.spheres}: $\bG\cong \prod_{p\in \cP}\bZ_p$ for some multiset $\cP$ of primes, so that
  \begin{equation*}
    \bG\cong \projlim_{\lambda\in \bZ_{\ge 0}^{\cP}}\left(\bG_{\lambda}:=\prod_{\cP}\bZ/p^{\lambda(p)}\right) 
  \end{equation*}
  with $\bZ_{\ge 0}^{\cP}$ ordered by pointwise domination (i.e. $\lambda\le \lambda'$ precisely when $\lambda(p)\le \lambda'(p)$ for all $p\in \cP$). One choice for the requisite actions $Y\circlearrowleft \bG_{\lambda}$ are the products of the individual $\bS^{2n-1}\circlearrowleft \bZ/p^m$ featuring in \Cref{cor:higher.spheres}.

  As to \Cref{item:th:cj1.holds.when:cpct} implying profiniteness, \Cref{eq:ghat.decomp} and \Cref{item:th:cj1.holds.when:cpct}'s permanence after passing to subgroups reduces the problem to proving \cite[Conjecture 2.3, Type 1]{MR3469159} for the specific compact group $\left(\wh{\bQ},+\right)$; this, we relegate to \Cref{th:qhat}. 
\end{th:cj1.holds.when}

\begin{remark}\label{re:s1.not.zpn}
  \cite[\S 4]{MR3803650} indicates the flaw in the original attempt \cite[Proposition 2.6]{1604.02173v2} at resolving the conjecture in its full generality: faulty use is made of a slightly misleading \cite[Claim 2.6]{zbMATH01746247} (misleading in that its later proof, in the form of \cite[Claim 4.9]{zbMATH01746247}, functions only for elementary abelian $p$-groups, whereas the claim refers to $p$-groups period).

  The crux of the argument supporting \Cref{th:qhat} is that that same proof strategy (to be recalled in sketch form presently) does function for $\wh{\bQ}$ (as opposed to $\bZ_p$); this is because analogues of the two referenced claims do hold for $\bS^1$, by contrast to $\bZ/p^n$, $n\ge 2$.
\end{remark}

\begin{theorem}\label{th:qhat}
  For a free action $X\circlearrowleft \left(\bG:=\left(\wh{\bQ},+\right)\right)$ on a non-empty compact Hausdorff space, there are no equivariant maps $X*\bG\to X$. 
\end{theorem}
\begin{proof}
  The argument runs as in (the valid portion of) the attempted \cite[proof of Theorem 2.1]{1604.02173v2}, with the appropriate alterations. The action being free, \cite[Theorem 2.9]{ell} ensures the existence of continuous $\chi_1$-eigenfunctions $f_i\in C(X)$ with
  \begin{equation*}
    \left\|\sum_{i=1}^n f_i g_i-1\right\|
    <\text{small }\varepsilon>0
    ,\quad
    \begin{gathered}
      g_i\in C(X)\\
      \bS^1
      \xrightarrow{\quad\chi_n:=(-)^n\quad}
      \bS^1
    \end{gathered}.
  \end{equation*}
  Iteration gives a $\bG$-equivariant map
  \begin{equation*}
    E_n\bG:= \bG^{*(n+1)}
    \lhook\joinrel\xrightarrow{\quad}
    X* E_n\bG
    \xrightarrow{\quad}
    X,
  \end{equation*}
  hence also a nowhere-vanishing $\bG$-map $E_n \bG\to \chi_1^{\oplus n}$ by restricting the functions $f_i$, $1\le i\le n$ along that composition. This violates \Cref{pr:no.triv.char}, finishing the proof. 
\end{proof}

\begin{proposition}\label{pr:no.triv.char}
  For an $n$-dimensional ($n\in \bZ_{>0}$) complex representation $\rho:\bG\circlearrowright E$ of $\bG:=\wh{\bQ}$ with no trivial summands, every $\bG$-equivariant map $E_n\bG:= \bG^{*(n+1)}\xrightarrow{f}E$ has a zero.
\end{proposition}
\begin{proof}
  $\rho$ factors through a Lie quotient $\bG\xrightarrowdbl{}\bS^1$ (all such are circles for $\bG=\wh{\bQ}$) which we may (and henceforth will) assume again has no trivial summands, and \cite[Lemma 3.3]{MR3803650} provides a nowhere-vanishing $\bS^1$-equivariant map
  \begin{equation*}
    E_n \bS^1\xrightarrow{\quad} E\cong \bigoplus_{i=1}^n \chi_{m_i}
    ,\quad
    \chi_m:=(-)^m\in \wh{\bS^1}. 
  \end{equation*}
  That this is impossible will follow via the bundle-theoretic machinery employed in \cite{zbMATH01746247}, per the remainder of the present proof.

  The \emph{universal bundle} \cite[\S 14.4]{td_alg-top}
  \begin{equation*}
    \varinjlim_m E_m\bS^1=:E\bS^1
    \xrightarrowdbl{\quad\pi\quad}
    B\bS^1:=E\bS^1/\bS^1
  \end{equation*}
  restricts to $\pi_m:=\pi|_{B_m\bS^1:=E_m\bS^1/\bS^1}$, $(2m+1)$-universal because $E_m \bS^1=\left(\bS^1\right)^{*(m+1)}$ is $2m$-connected by \cite[Lemma 2.3]{MR77932} (or simply because $E_m \bS^1\cong \bS^{2m+1}$ \cite[Proposition 2.23]{rs_pl_1982}). That $2m$-connectedness then ensures precisely as in the proof of \cite[Proposition 4.5]{zbMATH01746247} that the pullback
  \begin{equation*}
    H^*\left(B\bS^1,\bZ\right)
    \xrightarrow{\quad}
    H^*\left(B_n\bS^1,\bZ\right)
    \quad
    \left(\text{dual to }B_n\bS^1\lhook\joinrel\xrightarrow{\quad}B\bS^1\right)
  \end{equation*}
  is injective on $H^{2n}$, so the top \emph{Chern class} \cite[\S 14.2]{ms-cc} of the \emph{associated vector bundle} \cite[post Example 14.1.15]{td_alg-top} $\pi[\rho]$ vanishes if and only if that of the analogue $\pi_n[\rho]$ does. And in turn, that vanishing is equivalent (\cite[Theorem, bottom p. 284]{zbMATH01746247}; effectively \cite[Theorem 12.5]{ms-cc}) to the existence of the nowhere-vanishing map $E_n \bS^1\to E$ that is the focus of the present result.
  
  In light of the preceding discussion, it remains to argue that that top universal class $c_{n}\left(\pi[\rho]\right)$ is non-trivial. This, however, is immediate: if
  \begin{equation*}
    E \circlearrowleft \bS^1:\rho
    \cong
    \bigoplus_{i=1}^n
    \chi_{m_i}
    ,\quad
    \chi_m:=(-)^m\in \wh{\bS^1}
  \end{equation*}
  then the multiplicativity \cite[(14.7)]{ms-cc} of the \emph{total Chern class} implies that
  \begin{equation*}
    c_{n}\left(\pi[\rho]\right)
    =
    \alpha^n\prod_{i=1}^n m_i
  \end{equation*}
  for the degree-2 generator $\alpha$ of $H^*\left(\bC\bP^{\infty},\bZ\right)\cong \bZ[\alpha]$ \cite[Theorem 14.5]{ms-cc}. The $m_i$ being assumed non-zero, we are done.
\end{proof}

We record also the following consequence of \Cref{th:qhat}.

\begin{corollary}\label{cor:bdh1.fail.2.profin}
  If \cite[Conjecture 2.3, Type 1]{MR3469159} fails for the compact group $\bG$, then that group must be profinite. 
\end{corollary}
\begin{proof}
  We know \cite[Proposition 4.1]{MR4523296} the conjecture is valid for groups having non-trivial torsion, so we can assume $\bG$ torsion-free. Were it not profinite, it would then have non-trivial, torsion-free identity connected component $\bG_0$. The latter would then necessarily be abelian, as argued in the course of the proof of \cite[Proposition 1.10]{cdt_be-hs}; \Cref{th:qhat} thus applies to conclude that $\bG_0=\{1\}$.
\end{proof}


\addcontentsline{toc}{section}{References}

\def\polhk#1{\setbox0=\hbox{#1}{\ooalign{\hidewidth
  \lower1.5ex\hbox{`}\hidewidth\crcr\unhbox0}}}


\Addresses

\end{document}